\documentclass[12pt]{article}
\usepackage[utf8]{inputenc}
\usepackage{amsmath}
\usepackage{amssymb}
\usepackage{mathrsfs}
\usepackage{amsthm}
\usepackage{cite}
\usepackage{hyperref}
\usepackage{enumerate}
\usepackage{combelow}
\usepackage[left=0.7in,right=0.7in,top=0.7in,bottom=0.7in]{geometry}
\usepackage{titlesec}
\titleformat*{\section}{\large\bfseries}
\newtheorem{theorem}{Theorem}[section]
\newtheorem{lemma}[theorem]{Lemma}
\newtheorem{corollary}[theorem]{Corollary}
\newtheorem{definition}[theorem]{Definition}
\newtheorem{example}[theorem]{Example}
\newtheorem{proposition}[theorem]{Proposition}
\newtheorem{remark}[theorem]{Remark}
\numberwithin{equation}{section}
\title{Topological structure of quasi-partial b-metric-like spaces and some fixed point theorems}
\author{\large Anuradha Gupta$^1$ and Manu Rohilla$^2$\\ $^1$ {\small Department of Mathematics, Delhi College of Arts and Commerce,}\\ {\small University of Delhi, Netaji Nagar, New Delhi-110023, India.}\\{\small E-mail:dishna2@yahoo.in}\\{\small $^2$Department of Mathematics, University of Delhi, New Delhi-110007, India.}\\{\small E-mail:manurohilla25994@gmail.com}}
\date{}
\begin{document}
\maketitle

\begin{abstract}
The concept of quasi-partial b-metric-like spaces is being introduced and studied with the help of topology. Examples are also discussed to support the results. Some fixed point theorems are proved in the setting of quasi-partial b-metric-like spaces.
\end{abstract}

\section{Introduction and Preliminaries}
Metric space has several generalizations which have come in the latter part of 19th century and in the beginning of 20th century. Czerwick \cite{4} generalized the metric space by defining b-metric space in 1993. Later on Matthews \cite{11} studied partial metric space and obtained fixed point theorems on it. Shukla \cite{13} introduced the notion of partial b-metric space as a generalization of partial-metric and b-metric spaces. The concept of quasi-partial metric was introduced by Karapinar \cite{10}  and discussed some general fixed point theorems on it.\\
The topology of metric space plays a vital role for studying the notions of convergence and continuity. Dhage \cite{5} studied the topological properties of D-metric spaces. Gupta and Gautam \cite{8} generalized quasi-partial metric space and introduced the concept of quasi-partial b-metric space. In this paper we have introduced the concept quasi-partial b-metric-like spaces and examined the topological structure of quasi-partial b-metric-like spaces. Examples are also provided to illustrate the results obtained. The existence and uniqueness of fixed point of self mappings on quasi-partial b-metric-like space is discussed. The product of quasi-partial b-metric-like spaces is also obtained in this research article.\\
Alghamdi, Hussain and Salimi \cite{1} defined b-metric-like space in the following way:
\begin{definition}
 \emph{ A b-metric-like on a non-empty set $X$ is a function $D:X \times X \to  [0,\infty)$ such that for all $x,y,z \in X$ and a constant $s \geq 1$, the following conditions hold:
\begin{enumerate}
\item[(bl$_{1}$)] if $D(x,y)=0$, then $x=y$,
\item[(bl$_{2}$)] $D(x,y)=D(y,x)$,
\item[(bl$_{3}$)] $D(x,y) \leq s[D(x,z)+D(z,y)]$.
\end{enumerate}
The pair $(X,D)$ is called a b-metric-like space.}
\end{definition}
Gupta and Gautam \cite{8} introduced the concept of quasi-partial b-metric space as follows.
\begin{definition}
\emph{\cite{8} A quasi-partial b-metric on a non-empty set $X$ is a function $qp_b:X \times X \to  [0,\infty)$ such that for some real number $s \geq 1$ and all $x,y,z \in X$,  the following conditions hold:
\begin{itemize}
\addtolength{\itemindent}{0.4cm}
\item[(QPb$_{1}$)] if $qp_b(x,x)=qp_b(x,y)=qp_b(y,y$), then $x=y$,
\item[(QPb$_{2}$)] $qp_b(x,x) \leq qp_b(x,y)$,
\item[(QPb$_{3}$)] $qp_b(x,x) \leq qp_b(y,x)$,
\item[(QPb$_{4}$)] $qp_b(x,y) \leq s[qp_b(x,z)+qp_b(z,y)]-qp_b(z,z)$.
\end{itemize}
The pair $(X,qp_{b})$ is called a quasi-partial b-metric space. The number $s$ is called the coefficient of $(X,qp_b)$.}
\end{definition}

\section{Quasi-Partial b-metric-like space}
\begin{definition}
\emph{A quasi partial b-metric-like on a non-empty set $X$ is a function $qp_{bl}:X \times X \to  [0,\infty)$ such that for some real number $s \geq 1$ and all $x,y,z \in X$, the following conditions hold:
\begin{itemize}
\addtolength{\itemindent}{0.5cm}
\item[(QPbl$_{1}$)] if $qp_{bl}(x,y)=0$, then $x=y$,
\item[(QPbl$_{2}$)] $qp_{bl}(x,x) \leq qp_{bl}(x,y)$,
\item[(QPbl$_{3}$)] $qp_{bl}(x,x) \leq qp_{bl}(y,x)$,
\item[(QPbl$_{4}$)] $qp_{bl}(x,y) \leq s[qp_{bl}(x,z)+qp_{bl}(z,y)]-qp_{bl}(z,z)$.
\end{itemize}
The pair $(X,qp_{bl})$ is called a quasi-partial b-metric-like space. The number $s$ is called the coefficient of $(X,qp_{bl})$.}
\end{definition}
\begin{example}
\emph{Let $X=[0,1]$. Define $qp_{bl}: X \times X \rightarrow [0,\infty)$ as\\ $qp_{bl}(x,y)=\left\{\begin{array}{cc}
(x+y)^2,& \mbox{if}\thinspace \thinspace x \neq y,\\
0, &\mbox{if} \thinspace \thinspace x=y.
\end{array}
\right. $\\
Clearly (QPbl$_1$)-(QPbl$_3$) hold for all $x,y \in X$. As $0 \leq (x-y)^2+4xz+4zy$ for all $x,y,z \in X$. This implies  $(x+y)^2 \leq 2[(x+z)^2+(z+y)^2]-4z^2$. Therefore, (QPbl$_4$) holds. Thus, $(X,qp_{bl})$ is a quasi-partial b-metric-like space with $s=2$.}
\end{example} 
\begin{example}
\emph{Let $X=[0,\infty)$. Define $qp_{bl}: X \times X \rightarrow [0,\infty)$ as $qp_{bl}(x,y)=\max\{x,y\}+\mid x-y \mid$.\\
Obviously, (QPbl$_1$)-(QPbl$_3$) are satisfied. Let $x,y,z \in X$. If $x \leq y \leq z$, then
\begin{align*}
\max\{x,y\}+\mid x-y \mid &\leq y+\mid x-z \mid +\mid z-y \mid \\
&\leq \max\{x,z\}+\mid x-z \mid + \max\{z,y\}+\mid z-y \mid -z.
\end{align*}
Similarly, in all other cases (QPbl$_4$) is satisfied. Therefore, $(X,qp_{bl})$ is a quasi-partial b-metric-like space with $s=1$.}
\end{example}
\begin{example}
\emph{Let $X=[0,1]$. Define $qp_{bl}: X \times X \rightarrow [0,\infty)$ as $qp_{bl}(x,y)=$ $\mid x-y \mid +x$.\\
It is easily seen that $(X,qp_{bl})$ is a quasi-partial b-metric-like space with $s=1$.}
\end{example}
\begin{example}
\emph{Let $(X,d')$ be a metric space. Define $qp_{bl}: X \times X \rightarrow [0,\infty)$ as 
$qp_{bl}(x,y)= (d'(x,y))^q$ where $q >1$.\\
Obviously, (QPbl$_1$)-(QPbl$_3$) are satisfied. For all $x,y,z \in X$ we have $ qp_{bl}(x,y)$  $\leq (d'(x,z)+d'(z,y))^q \leq  2^{q-1}[(d'(x,z))^q+(d'(z,y))^q]-qp_{bl}(z,z)$.
Therefore, $(X,qp_{bl})$ is a quasi-partial b-metric-like space with $s=2^{q-1}$.}
\end{example}
\begin{definition}
\emph{ A quasi-partial b-metric-like space $(X,qp_{bl})$ is said to be symmetric if $qp_{bl}(x,y)=qp_{bl}(y,x)$ for all $x,y \in X$.}
\end{definition}
Every quasi-partial b-metric space is quasi-partial b-metric-like space. But the converse need not be true as shown in the following example:\\
Let $X=\{0,1,2\}$. Define $d:X \times X \rightarrow [0,\infty)$ as\\ 
$d(0,0)=0, \hspace{2cm} d(0,1)=1, \hspace{2cm} d(0,2)=1,$\\
$d(1,0)=2, \hspace{2cm} d(1,1)=\frac{1}{2}, \hspace{2cm} d(1,2)=\frac{1}{2},$\\
$d(2,0)=3, \hspace{2cm} d(2,1)=3, \hspace{2cm} d(2,2)=\frac{1}{2}.$\\
Then $(X,d)$ is quasi-partial b-metric-like space with $s= 1$. But $(X,d)$ is not quasi-partial b-metric space as $d(1,1)=d(1,2)=d(2,2)$ but $1 \neq 2$. 
\section{Quasi-partial b-metric-like topology}
Mustafa \emph{et al.} \cite{12} discussed the topological structure of partial b-metric space. In this section we define topology on quasi-partial b-metric-like space and its  topological properties are studied. 
\begin{definition}
\emph{Let $(X,qp_{bl})$ be quasi-partial b-metric-like space. Then for $x_0 \in X$, $\epsilon >0$ the ball centered at $x_0$ and radius $\epsilon$ is defined as
\begin{align*}
B_{qp_{bl}} (x_0; \epsilon)=&\{ y \in X: qp_{bl}(x_0,y) < qp_{bl}(x_0,x_0)+ \epsilon \thinspace \thinspace \mbox{and}\\
& \hspace{0.2cm} qp_{bl}(y,x_0) < qp_{bl}(x_0,x_0)+ \epsilon \}.
\end{align*}}
\end{definition}
\begin{example}
\emph{Let $X=[0,1]$. Define $qp_{bl}: X \times X \rightarrow [0,\infty)$ as $qp_{bl}(x,y)=\max\{x,y\}+\mid x-y \mid$. Then $(X,qp_{bl})$ is a quasi-partial b-metric-like space with $s=1$.
The ball centered at $0$ and radius $1$ is given by\\
$B_{qp_{bl}}(0;1) = \{y \in [0,1]: qp_{bl}(0,y) < qp_{bl}(0,0)+1$ and $qp_{bl}(y,0) < qp_{bl}(0,0)+1\} = \{ y \in [0,1]: y+\mid y \mid  <1 \} =\Big[0,\frac{1}{2}\Big)$.}
\end{example}
\begin{example}
\emph{Let $X=[0,1]$. Define $qp_{bl}: X \times X \rightarrow [0, \infty)$ as $qp_{bl}: X \times X \rightarrow [0,\infty)$ as $qp_{bl}(x,y)=\left\{\begin{array}{cc}
(x+y)^2,& \mbox{if}\thinspace \thinspace  x \neq y,\\
0, &\mbox{if} \thinspace \thinspace x=y.
\end{array}
\right. $\\
 Then $(X,qp_{bl})$ is a quasi-partial b-metric-like space with $s=2$.
The ball centered at $0$ and radius $\frac{1}{2}$ is given by\\
 $B_{qp_{bl}} \Big(0;\frac{1}{2}\Big) = \Big\{y \in [0,1]: qp_{bl}(0,y) < qp_{bl}(0,0)+\frac{1}{2}$ and $qp_{bl}(y,0) < qp_{bl}(0,0)+\frac{1}{2}\Big\}= \Big\{ y \in [0,1]: y^2 < \frac{1}{2}\Big\}=\Big[0,\frac{1}{\sqrt{2}}\Big)$.}
\end{example}
\begin{proposition}\label{proposition3.4}
Let $(X,qp_{bl})$ be a quasi-partial b-metric-like space with coefficient $s \geq 1$, then for any $x \in X$ and $\epsilon>0$, and if $y \in B_{qp_{bl}}(x;\epsilon)$, then there exists $\delta >0$ such that $B_{qp_{bl}}(y; \delta)\subseteq B_{qp_{bl}}(x;\epsilon)$.
\end{proposition}
\begin{proof}
Suppose that $y \in B_{qp_{bl}}(x;\epsilon)$. If $y=x$ take $\delta = \epsilon$. Let $y \neq x$, then $qp_{bl}(x,y) \neq 0$. Now we consider the following two cases:\\ 
\textbf{Case-1} If $qp_{bl}(x,x)= qp_{bl}(x,y)= qp_{bl}(y,y).$\\
\textsf{Subcase-1} If $s=1$. Take $\delta=\epsilon$. Suppose that $z \in B_{qp_{bl}}(y;\delta)$ then $qp_{bl}(z,y)< qp_{bl}(y,y)+\delta$ and $qp_{bl}(y,z) < qp_{bl}(y,y)+\delta$. We observe that $qp_{bl}(x,z) \leq qp_{bl}(x,y)+ qp_{bl}(y,z)-qp_{bl}(y,y)
 < qp_{bl}(x,y)+qp_{bl}(y,y)+\delta-qp_{bl}(y,y)=qp_{bl}(x,x)+\epsilon$. Similarly, $qp_{bl}(z,x)  \leq qp_{bl}(x,x)+\epsilon $. Therefore, $B_{qp_{bl}}(y; \delta)\subseteq B_{qp_{bl}}(x; \epsilon)$.\\
\textsf{Subcase-2} If $s>1$. Consider the set
\[
A=\Big\{n \in \mathbb{N} : qp_{bl}(x,x)>\frac{\epsilon}{2s^{n+1}(2s-1)}\Big\}.
\]
By the Archimedean property, $A$ is a non-empty set. Then by well-ordering principle, $A$ has a least element say $m$. This implies that $m-1 \notin A$. This gives
\begin{eqnarray}\label{3.1}
 qp_{bl}(x,x) \leq \frac{\epsilon}{2s^{m}(2s-1)}.
 \end{eqnarray}
Take $\delta = \frac{\epsilon}{2s^{m+1}}$. Suppose that  $z \in B_{qp_{bl}}(y;\delta)$, then $qp_{bl}(z,y)< qp_{bl}(y,y)+\delta$ and $qp_{bl}(y,z) < qp_{bl}(y,y)+\delta$. As $qp_{bl}(z,x) \leq s[qp_{bl}(z,y)+ qp_{bl}(y,x)]- qp_{bl}(y,y)\leq s[qp_{bl}(y,x)+qp_{bl}(y,x)]+ s\delta =qp_{bl}(x,x)+(2s-1)qp_{bl}(x,x)+\frac{ s \epsilon}{2s^{m+1}}$. Using \eqref{3.1} we get, $qp_{bl}(z,x) \leq qp_{bl}(x,x)+\frac{\epsilon}{2s^m}+\frac{s \epsilon}{2s^{m+1}}< qp_{bl}(x,x)+\epsilon$. Similarly, $qp_{bl}(x,z)  \leq qp_{bl}(x,x)+\epsilon $. Therefore, $B_{qp_{bl}}(y; \delta)\subseteq B_{qp_{bl}}(x; \epsilon)$.\\
\textbf{Case-2} If $qp_{bl}(x,x) \leq qp_{bl}(x,y)$ and $qp_{bl}(x,x)<qp_{bl}(y,x)$.\\
\textsf{Subcase-1} If $s=1$. Consider the set
$$B=\Big\{n \in \mathbb{N} : qp_{bl}(x,y)+qp_{bl}(y,x)-qp_{bl}(x,x)>\frac{\epsilon}{2^{n+2}}\Big\}.$$
By the Archimedean property, $B$ is a non-empty set. Then by well-ordering principle, $B$ has a least element say $p$. This implies that $p-1 \notin B$. This gives 
\begin{eqnarray}\label{3.2}
 qp_{bl}(x,y)+qp_{bl}(y,x)-qp_{bl}(x,x) \leq \frac{\epsilon}{2^{p+1}}. 
 \end{eqnarray}
Take $\delta = \frac{\epsilon}{2^{p+1}}$. Suppose that $z \in B_{qp_{bl}}(y;\delta)$, then $ qp_{bl}(z,y)< qp_{bl}(y,y)+\delta$ and $qp_{bl}(y,z) < qp_{bl}(y,y)+\delta$. As $qp_{bl}(z,x) \leq qp_{bl}(z,y)+qp_{bl}(y,x)-qp_{bl}(y,y)<qp_{bl}(y,y)+\delta +qp_{bl}(y,x)-qp_{bl}(y,y)$. Using \eqref{3.2} we get, $qp_{bl}(z,x)
\leq \delta +\frac{\epsilon}{2^{p+1}}+qp_{bl}(x,x) \leq qp_{bl}(x,x)+\epsilon$. Similarly, $qp_{bl}(x,z)  \leq qp_{bl}(x,x)+\epsilon $. Therefore, $B_{qp_{bl}}(y; \delta)\subseteq B_{qp_{bl}}(x; \epsilon)$.\\
\textsf{Subcase-2} If $s>1$. Consider the set
$$C=\Big\{n \in \mathbb{N} : qp_{bl}(x,y)+qp_{bl}(y,x)-\frac{1}{s}qp_{bl}(x,x)>\frac{\epsilon}{2s^{n+2}}\Big\}.$$
By the Archimedean property, $C$ is a non-empty set. Then by well-ordering principle, $C$ has a least element say $r$. This implies that $ r-1 \notin C$. This gives 
\begin{eqnarray}\label{3.3}
 qp_{bl}(x,y)+qp_{bl}(y,x)-\frac{1}{s}qp_{bl}(x,x) \leq \frac{\epsilon}{2s^{r+1}}.
 \end{eqnarray}
Take $\delta = \frac{\epsilon}{2s^{r+1}}$. Suppose that $z \in B_{qp_{bl}}(y;\delta)$, then $qp_{bl}(z,y)< qp_{bl}(y,y)+\delta$ and $qp_{bl}(y,z) < qp_{bl}(y,y)+\delta$. As $qp_{bl}(z,x) \leq s[qp_{bl}(z,y)+qp_{bl}(y,x)]-qp_{bl}(y,y) <s[qp_{bl}(y,y)+\delta +qp_{bl}(y,x)]$. Using \eqref{3.3} we get, $qp_{bl}(z,x) \leq s[qp_{bl}(y,y)+\delta +\frac{\epsilon}{2s^{r+1}}+\frac{1}{s}qp_{bl}(x,x)-qp_{bl}(x,y)] \leq s[qp_{bl}(x,y)+\delta +\frac{\epsilon}{2s^{r+1}}+\frac{1}{s}qp_{bl}(x,x)-qp_{bl}(x,y)]\leq qp_{bl}(x,x)+\epsilon$. Similarly, $qp_{bl}(x,z)  \leq qp_{bl}(x,x)+\epsilon $.  
Therefore, $B_{qp_{bl}}(y; \delta)\subseteq B_{qp_{bl}}(x; \epsilon)$.
\end{proof}
The family of all $qp_{bl}$-balls is denoted by $\mathcal{B}=\{B_{{qp}_{bl}}(x;\epsilon): x \in X, \epsilon>0\}$. In the following result it is shown that $\mathcal{B}$ is the base of topology $\tau_{{qp}_{bl}}$ on $X$, where $\tau_{{qp}_{bl}}$ is the quasi-partial b-metric-like topology.
\begin{theorem}
The collection $\mathcal{B}=\{B_{{qp}_{bl}}(x;\epsilon): x \in X, \epsilon>0\}$ of all the open balls forms a basis for a topology $\tau_{{qp}_{bl}}$ on $X$.
\end{theorem} 
\begin{proof}
It is enough to show that the collection $\mathcal{B}$ satisfies the following two conditions:
\begin{itemize}
\item[(i)] $X \subseteq \Big(\bigcup\limits_{\substack{x \in X,\\\epsilon>0}}B_{{qp}_{bl}}(x;\epsilon)\Big)$ and
\item[(ii)] if for some $x,y \in X$, $a \in B_{{qp}_{bl}}(x;\epsilon_1)\cap B_{{qp}_{bl}}(y; \epsilon_2)$ be an arbitrary point, then there is a ball $B_{{qp}_{bl}}(a;\delta)$ for some $\delta >0$ such that
$B_{{qp}_{bl}}(a;\delta) \subseteq B_{{qp}_{bl}}(x;\epsilon_1) \cap B_{{qp}_{bl}}(y;\epsilon_2)$.
\end{itemize}
(i) Suppose that  $x\in X$. Clearly $x\in B_{{qp}_{bl}}(x;\epsilon)$ for $\epsilon>0$. This gives $x\in B_{{qp}_{bl}}(x;\epsilon) \subseteq \Big(\bigcup\limits_{\substack{x \in X,\\\epsilon>0}}B_{{qp}_{bl}}(x;\epsilon)\Big)$.\\
(ii) Suppose that $a \in B_{{qp}_{bl}}(x; \epsilon_1)\cap B_{{qp}_{bl}}(y; \epsilon_2)$.
By Proposition \ref{proposition3.4}, there exist $\delta_1,\delta_2 >0$ such that  $B_{{qp}_{bl}}(a;\delta_1) \subseteq B_{{qp}_{bl}}(x;\epsilon_1)$ and $B_{{qp}_{bl}}(a;\delta_2) \subseteq B_{{qp}_{bl}}(x;\epsilon_2)$. Choose $\delta= \min \{\delta_1,\delta_2\}$. Suppose that  $z \in B_{{qp}_{bl}}(a;\delta)$, then ${qp}_{bl}(a,z)< {qp}_{bl}(a,a)+\delta$ and ${qp}_{bl}(z,a)< {qp}_{bl}(a,a)+\delta$. This gives
\begin{eqnarray}\label{3.4}
 {qp}_{bl}(a,z)< {qp}_{bl}(a,a)+\delta_1\quad \mbox{and} \quad  {qp}_{bl}(z,a)< {qp}_{bl}(a,a)+\delta_1, 
 \end{eqnarray}
and
\begin{eqnarray}\label{3.5}
{qp}_{bl}(a,z)< {qp}_{bl}(a,a)+\delta_2 \quad \mbox{and} \quad {qp}_{bl}(z,a)< {qp}_{bl}(a,a)+\delta_2.
\end{eqnarray}
From \eqref{3.4} and \eqref{3.5} we get, $B_{{qp}_{bl}}(a;\delta) \subseteq B_{{qp}_{bl}}(x;\epsilon_1) \cap B_{{qp}_{bl}}(y;\epsilon_2)$.
\end{proof}
Therefore, the quasi-partial b-metric-like space $(X, qp_{bl})$ together with a topology $\tau_{{qp}_{bl}}$ is called a quasi-partial b-metric-like topological space and $\tau_{{qp}_{bl}}$ is called a quasi-partial b-metric-like topology on $X$.

 A space $X$ is called \emph{Hausdroff}(or $\mathcal{T}_2$) if for every pair of distinct points $x,y \in X$, there exist disjoint open sets $U$ and $V$ with $x \in U$ and $y \in V$. A space $X$ is called $\mathcal{T}_1$ if for each pair of distinct points $x,y \in X$, there is an open set containing $x$ but not $y$, and another open set containing $y$ but not $x$. A space $X$ is called $\mathcal{T}_0$ if for any two distinct points of $X$, there is an open set which contains one point but not the other.
\begin{remark}\label{remark1}
\emph{A quasi-partial b-metric-like space need not be a $T_0$ space. Consider the following example:}
\end{remark}
Let $X=\{0,1,2\}$. Define $qp_{bl}:X \times X \rightarrow [0,1]$ as\\ 
$qp_{bl}(0,0)=0, \hspace{2cm} qp_{bl}(0,1)=1, \hspace{2cm} qp_{bl}(0,2)=1,$\\
$qp_{bl}(1,0)=1, \hspace{2cm} qp_{bl}(1,1)=1, \hspace{2cm} qp_{bl}(1,2)=1,$\\
$qp_{bl}(2,0)=1, \hspace{2cm} qp_{bl}(2,1)=1, \hspace{2cm} qp_{bl}(2,2)=1.$ \\
Then $(X,qp_{bl})$ is a quasi-partial b-metric-like space with $s=1$. Consider 
\begin{align*}
 B_{{qp}_{bl}}(0;\epsilon)&= \{ y \in X : qp_{bl}(0,y) < qp_{bl}(0,0) + \epsilon \thinspace \mbox{and} \thinspace qp_{bl}(y,0) < qp_{bl}(0,0) + \epsilon \}\\
&=\left\{\begin{array}{ll}
\{0\},& \mbox{if}\thinspace \thinspace  \epsilon \leq 1,\\
\{0,1,2\},&\mbox{if} \thinspace \thinspace \epsilon >1.
\end{array}
\right.
\end{align*}
\begin{align*}
B_{{qp}_{bl}}(1;\epsilon)&= \{ y \in X : qp_{bl}(1,y) < qp_{bl}(1,1) + \epsilon \thinspace \mbox{and} \thinspace qp_{bl}(y,1) < qp_{bl}(1,1) + \epsilon \}\\
&=\{0,1,2\}.
\end{align*}
\begin{align*}
B_{{qp}_{bl}}(2;\epsilon)&= \{ y \in X : qp_{bl}(2,y) < qp_{bl}(2,2) + \epsilon \thinspace \mbox{and} \thinspace qp_{bl}(y,2) < qp_{bl}(2,2) + \epsilon \}\\
&=\{0,1,2\}.
\end{align*}
Therefore, $\tau_{{qp}_{bl}}= \{ \phi, \{0\}, X \}$. We see  $1,2 \in X$ but there does not exist an open set containing 1 but not 2 and there does not exist an open set containing 2 but not 1. Thus, $(X, qp_{bl})$ is not a $T_0$ space.
\begin{theorem}
A quasi-partial b-metric-like space need not be a $T_1$ or $T_2$ space.
\end{theorem}
\begin{definition}
\emph{Let $(X,qp_{bl})$ be a quasi-partial b-metric-like space. Then\\
(i) A sequence $\{x_n\} \subseteq X$ \textit{converges} to $x \in X$ if and only if $$\lim\limits_{n\rightarrow \infty }qp_{bl}(x_n,x)=qp_{bl}(x,x)=\lim\limits_{n\rightarrow \infty }qp_{bl}(x,x_n).$$
(ii) A sequence $\{x_n\} \subseteq X$ is called a \textit{Cauchy} sequence if and only if
 $$\lim\limits_{n,m \rightarrow \infty} qp_{bl}(x_n,x_m) \thinspace \mbox{and} \thinspace \lim\limits_{n,m \rightarrow \infty} qp_{bl}(x_m,x_n) \thinspace \mbox{exist (and are finite)}.$$
 (iii) A quasi-partial b-metric-like space $(X,qp_{bl})$ is said to be \textit{complete} if every Cauchy sequence $\{x_n\} \subseteq X$ converges with respect to $\tau_{{qp}_{bl}}$ to a point $x \in X$ such that 
 $$\lim\limits_{n,m \rightarrow \infty} qp_{bl}(x_n,x_m) =qp_{bl}(x,x)=\lim\limits_{n,m \rightarrow \infty} qp_{bl}(x_m,x_n).$$} 
\end{definition}
\begin{remark}
\emph{In Remark \ref{remark1},  let $x_n=1$ for each $ n \in \mathbb{N}$. Then $ \lim\limits_{n \rightarrow \infty} qp_{bl}(x_n,1)$ $=qp_{bl}(1,1)= \lim\limits_{n \rightarrow \infty} qp_{bl}(1,x_n)$. Therefore, $x_n \rightarrow 1$. Also, $ \lim\limits_{n \rightarrow \infty} qp_{bl}(x_n,2)=qp_{bl}(1,2)= qp_{bl}(2,2)=\lim\limits_{n \rightarrow \infty} qp_{bl}(2,x_n)$. Therefore, $x_n \rightarrow 2$.
Hence, in a quasi-partial b-metric-like space the limit of a  sequence is not necessarily unique.}
\end{remark}
\begin{remark}
 \emph{In a quasi-partial b-metric-like space $(X, qp_{bl})$ the function $qp_{bl}$ need not be continuous in any of its variables. The following example illustrates this fact.}
\end{remark}
\begin{example}
\emph{Let $X= \mathbb{N} \cup \{+\infty\}$. Define $qp_{bl}:X \times X \rightarrow [0,\infty )$ as
\[
qp_{bl}(x,y):=\left\{\begin{array}{ll}
0, &x=+\infty, y=+ \infty, \\
\mid \frac{1}{x} -\frac{1}{y} \mid , &\mbox{both}\thinspace  x \thinspace \mbox{and} \thinspace y \thinspace  \mbox{are odd}, \\
\frac{1}{x}, & x \thinspace \mbox{is odd and}\thinspace  y \thinspace \mbox{is} \thinspace \infty, \\
\frac{1}{2x}, & x \thinspace  \mbox{is} \thinspace \infty \thinspace \mbox{and}\thinspace  y \thinspace \mbox{is odd,}\\
1, & \mbox{otherwise.}
\end{array}
\right.
\]
Then $(X,qp_{bl})$ is a quasi-partial b-metric-like space with $s=2$.
Let $x_n=2n+1$ for each $n \in \mathbb{N}$. Then $qp_{bl}(x_n,2)=1=qp_{bl}(2,2)$ and $qp_{bl}(2,x_n)=1=qp_{bl}(2,2).$  Therefore, $x_n \rightarrow 2$.
But $qp_{bl}(x_n,3) \rightarrow \frac{1}{3}$ and $qp_{bl}(2,3)=1$.}
\end{example}
For a quasi-partial b-metric-like space $(X,qp_{bl})$ the function $D: X \times X \rightarrow [0, \infty)$ defined by $D(x,y)=qp_{bl}(x,y)+qp_{bl}(y,x)$ is a b-metric-like on $X$.
\begin{definition}
 \emph{\cite{1} Let $(X,qp_{bl})$ be a quasi-partial b-metric-like space and $(X,D)$ be the corresponding b-metric-like space. Then the open ball centered at $x_0 \in X$ and radius $\epsilon>0$ is defined as 
$$B_{D}(x_0;\epsilon)=\{y\in X: \mid D(x_0,y)-D(x_0,x_0) \mid < \epsilon \}.$$}
\end{definition}
\begin{proposition}
Let $(X,qp_{bl})$ be a quasi-partial b-metric-like space with coefficient $s \geq 1$, then for all $x\in X$ and $\epsilon>0$
$$ B_{{qp}_{bl}}\Big(x;\frac{\epsilon}{2}\Big) \subseteq  B_{D}(x;\epsilon) \subseteq  B_{{qp}_{bl}}(x;\delta),$$
where $\delta= s[\epsilon+2qp_{bl}(x,x)]$.
\end{proposition}
\begin{proof}
Suppose that $z\in  B_{{qp}_{bl}} \Big(x;\frac{\epsilon}{2}\Big)$, then
\begin{eqnarray}\label{3.6}
 qp_{bl}(x,z)<qp_{bl}(x,x)+\frac{\epsilon}{2} \thinspace \thinspace \mbox{and} \thinspace \thinspace qp_{bl}(z,x)<qp_{bl}(x,x)+\frac{\epsilon}{2}. 
 \end{eqnarray}
Using \eqref{3.6} we get, 
$\mid D(x,z)-D(x,x)\mid < \epsilon$. Therefore, $ B_{{qp}_{bl}}\Big(x;\frac{\epsilon}{2}\Big) \subseteq B_{D}(x;\epsilon)$. Suppose that $z \in B_D(x;\epsilon)$, then $\mid D(x,z)-D(x,x) \mid < \epsilon$. This gives 
\begin{eqnarray}\label{3.7}
qp_{bl}(x,z)<\epsilon+2qp_{bl}(x,x)-qp_{bl}(z,x),
\end{eqnarray}
and
 \begin{eqnarray}\label{3.8}
 qp_{bl}(z,x)<\epsilon+2qp_{bl}(x,x)-qp_{bl}(x,z).
 \end{eqnarray}
Since $qp_{bl}(x,z) \leq s[qp_{bl}(x,x)+qp_{bl}(x,z)]-qp_{bl}(x,x) \leq s[qp_{bl}(x,x)+qp_{bl}(x,z)]$. 
Using \eqref{3.7} we get, $qp_{bl}(x,z)<s[\epsilon+2qp_{bl}(x,x)]+s[qp_{bl}(x,x)-qp_{bl}(z,x)]\leq s[\epsilon+2qp_{bl}(x,x)]=\delta$. Similarly, by \eqref{3.8} we have $qp_{bl}(z,x)< \delta$. Therefore, $B_D(x;\epsilon) \subseteq B_{{qp}_{bl}}(x;\delta)$ where $\delta= s[\epsilon+2qp_{bl}(x,x)]$.
\end{proof}
\begin{definition}
\emph{\cite{1} Let $(X,D)$ be a b-metric-like space. Then\\
(i) A sequence $\{x_n\} \subseteq X$ \textit{converges} to $x \in X$ if and only if $\lim\limits_{n\rightarrow \infty }D(x_n,x)=D(x,x)$.\\
(ii) A sequence $\{x_n\} \subseteq X$ is called a \textit{Cauchy sequence} if and only if\\ $\lim\limits_{n,m \rightarrow \infty} D(x_n,x_m)$  exists (and is finite).\\
 (iii) A b-metric-like space $(X,D)$ is said to be \textit{complete} if every Cauchy sequence $\{x_n\} \subseteq X$ converges with respect to $\tau_{D}$ to a point $x \in X$ such that 
 $\lim\limits_{n,m \rightarrow \infty} D(x_n,x_m) =D(x,x)=\lim\limits_{n \rightarrow \infty} D(x_n,x)$.}
\end{definition}
\begin{proposition}\label{proposition3.15}
Let $(X,qp_{bl})$ be a quasi-partial b-metric-like space with coefficient $s \geq 1$ and $(X,D)$ be the corresponding b-metric-like space. A sequence  $\{x_n\}$ is  Cauchy in $(X,qp_{bl})$ if and only if $\{x_n\}$ is  Cauchy in $(X,D)$. 
\end{proposition}
\begin{proof}
Let $\{x_n\}$ be a Cauchy sequence in $(X,qp_{bl})$. Then $\lim\limits_{n,m \rightarrow \infty} qp_{bl}(x_n,x_m)$ and $\lim\limits_{n,m \rightarrow \infty} qp_{bl}(x_m,x_n)$ exist and are finite. Therefore, there exist $\alpha, \beta \geq 0$ such that for every $\epsilon >0$, there are $N_1,N_2 \in \mathbb{N}$ such that
\[
 \vert qp_{bl}(x_n,x_m)-\alpha \vert  < \frac{\epsilon}{2} \quad \mbox{for all}\thinspace \thinspace n,m \geq N_1,
 \]
and
\[
 \vert qp_{bl}(x_m,x_n)-\beta \vert  < \frac{\epsilon}{2} \quad \mbox{for all}\thinspace \thinspace n,m \geq N_2.
 \]
Let $N=\max\{N_1,N_2\}$. Then for all $n,m \geq N$, we have $\vert D(x_n,x_m)-(\alpha + \beta) \vert  < \epsilon$ for all $n,m \geq N$. Therefore, $\lim\limits_{n,m \rightarrow \infty} D(x_n,x_m) $ exists and is finite. Hence, $\{x_n\}$ is Cauchy in $(X,D)$.

Conversely, let $\{x_n\}$ be a Cauchy sequence in $(X,D)$. Then \\ $\lim\limits_{n,m \rightarrow \infty} D(x_n,x_m)$ exists and is finite. Therefore, there exists $\eta \geq 0$ such that for every $\epsilon >0$, there is $N_3 \in \mathbb{N}$ such that 
\[
\vert D(x_n,x_m) -\eta \vert < \frac{\epsilon}{2} \quad \mbox{for all} \thinspace \thinspace n,m \geq N_3.
\] As $D(x_n,x_n) \leq s[D(x_n,x_m)+D(x_m,x_n)]=2s D(x_n , x_m)$. This implies that $\lim\limits_{n\rightarrow \infty} D(x_n,x_n)$ exists and is finite. Since $D(x_n,x_n)=2 qp_{bl}(x_n,x_n)$. Therefore, $\lim\limits_{n\rightarrow \infty} qp_{bl}(x_n,x_n)$ exists and is finite. Then there exists $\zeta \geq 0$ such that for every $\epsilon >0$, there is $N_4 \in \mathbb{N}$ such that 
\[
\vert qp_{bl}(x_n,x_n) -\zeta \vert < \frac{\epsilon}{2} \quad \mbox{for all} \thinspace \thinspace n\geq N_4.
\] 
Let $N'=\max\{N_3,N_4\}$. Then for all $n,m \geq N'$, we have $\vert qp_{bl}(x_n,x_m) - (\eta-\zeta) \vert = \vert qp_{bl}(x_n,x_m)+qp_{bl}(x_n,x_n)-qp_{bl}(x_n,x_n)-\eta+\zeta \vert \leq \vert qp_{bl}(x_n,x_m)+qp_{bl}(x_m,x_n)-qp_{bl}(x_n,x_n)-\eta+\zeta \vert < \epsilon$. Similarly, $\lim\limits_{n,m\rightarrow \infty} qp_{bl}(x_m,x_n)$ exists and is finite. Hence, $\{x_n\}$ is Cauchy in $(X,qp_{bl})$.
\end{proof}
\begin{theorem}
Let $(X,qp_{bl})$ be a quasi-partial b-metric-like space with coefficient $s \geq 1$ and $(X,D)$ be the corresponding b-metric-like space. Then $(X,D)$ is complete if and only if $(X,qp_{bl})$ is complete.
\end{theorem}
\begin{proof}
Suppose that $(X,D)$ is complete. Let $\{x_n\}$ be a Cauchy sequence in $(X,qp_{bl})$. Then by Proposition \ref{proposition3.15}, $\{x_n\}$ is Cauchy in $(X,D)$. Therefore, there exists $x \in X$ such that $ \lim\limits_{n,m \rightarrow \infty} D(x_n,x_m) =D(x,x)=\lim\limits_{n \rightarrow \infty} D(x_n,x)$.
As $\lim\limits_{n \rightarrow \infty} D(x_n,x)=D(x,x)$. This implies that $\lim\limits_{n \rightarrow \infty} [qp_{bl}(x_n,x)+qp_{bl}(x,x_n)-qp_{bl}(x,x)-qp_{bl}(x,x)]=0$. Since $qp_{bl}(x,x) \leq qp_{bl}(x_n,x)$ and $qp_{bl}(x,x) \leq qp_{bl}(x,x_n)$. Therefore, $\lim\limits_{n \rightarrow \infty} qp_{bl}(x_n,x) = qp_{bl}(x,x)$ and  $\lim\limits_{n \rightarrow \infty} qp_{bl}(x,x_n) = qp_{bl}(x,x).$\\
\textbf{Case-1} If $qp_{bl}(x,x)=0$. Since $qp_{bl}(x_n,x_m) \leq s[ qp_{bl}(x_n,x)+qp_{bl}(x,x_m)]-qp_{bl}(x,x)$. Thus, $\lim\limits_{n,m \rightarrow \infty} qp_{bl}(x_n,x_m) =0$. Similarly, $\lim\limits_{n,m \rightarrow \infty} qp_{bl}(x_m,x_n) =0$. \\
\textbf{Case-2} If $qp_{bl}(x,x)>0$. Consider the set
$$D=\Big\{p \in \mathbb{N} : qp_{bl}(x,x)>\frac{\epsilon}{4s^{p+1}(2s-1)}\Big\}.$$
By the Archimedean property, $D$ is a non-empty set. Then by well-ordering principle, $D$ has a least element say $q$. This implies that $q-1 \notin D$. This gives
\begin{eqnarray}\label{3.9}
qp_{bl}(x,x) \leq \frac{\epsilon}{4s^{q}(2s-1)}.
\end{eqnarray}
Since  $\lim\limits_{n \rightarrow \infty} qp_{bl}(x_n,x)=qp_{bl}(x,x)$. Then for $\delta= \frac{\epsilon}{4s^{q+1}}$, there exists $N_1 \in \mathbb{N}$ such that
$$\mid qp_{bl}(x_n,x)- qp_{bl}(x,x) \mid < \delta \quad \mbox{for all} \thinspace \thinspace n \geq N_1.$$
Also,  $\lim\limits_{m \rightarrow \infty} qp_{bl}(x,x_m)=qp_{bl}(x,x)$. Then for $\delta= \frac{\epsilon}{4s^{q+1}}$, there exists $N_2 \in \mathbb{N}$ such that
$$\mid qp_{bl}(x,x_m)- qp_{bl}(x,x) \mid < \delta \quad \mbox{for all} \thinspace \thinspace n \geq N_2.$$
Let $N=$ max $\{N_1,N_2\}$. Then for all $n,m \geq N$, we have $qp_{bl}(x_n,x_m)  \leq s[qp_{bl}(x_n,x)+qp_{bl}(x,x_m)]-qp_{bl}(x,x)< 2s\delta + qp_{bl}(x,x) +(2s-1) qp_{bl}(x,x)$. Using \eqref{3.9} we get, $qp_{bl}(x_n,x_m)\leq \frac{2s \epsilon}{4s^{q+1}}+qp_{bl}(x,x)+\frac{ \epsilon}{4s^{q}}< \epsilon+ qp_{bl}(x,x)$. This implies  that $\lim\limits_{n,m \rightarrow \infty} qp_{bl}(x_n,x_m) \leq qp_{bl}(x,x)$. A similar argument shows that   $\lim\limits_{n,m \rightarrow \infty} qp_{bl}(x_m,x_n) \leq qp_{bl}(x,x)$. Since $ D(x,x)=\lim\limits_{n,m \rightarrow \infty} D(x_n,x_m)$ therefore, $\lim\limits_{n,m \rightarrow \infty} qp_{bl}(x_n,x_m) = qp_{bl}(x,x)= \lim\limits_{n,m \rightarrow \infty} qp_{bl}(x_m,x_n) $. Hence, $(X, qp_{bl})$ is complete.

Conversely, suppose that $(X, qp_{bl})$ is complete. Let $\{y_n\}$ be a Cauchy sequence in $(X,D)$. Then by Proposition \ref{proposition3.15}, $\{y_n\}$ is Cauchy in $(X, qp_{bl})$. Therefore, there exists $y \in X$ such that $y_n \to y$ and $ \lim\limits_{n,m \rightarrow \infty} qp_{bl}(y_n,y_m) =qp_{bl}(y,y)=\lim\limits_{n,m \rightarrow \infty} qp_{bl}(y_m,y_n)$. Then for $\epsilon >0$, there exist $N_3$, $N_4 \in \mathbb{N}$ such that 
\[
\vert qp_{bl}(y_n,y_m) -qp_{bl}(y,y)\vert < \frac{\epsilon}{2} \quad \mbox{for all} \thinspace \thinspace n,m \geq N_3,
\]
and  
\[
\vert qp_{bl}(y_m,y_n) -qp_{bl}(y,y)\vert < \frac{\epsilon}{2} \quad \mbox{for all} \thinspace \thinspace n,m \geq N_4.
\]
Let $N'=\max \{N_3,N_4\}$. Then for all $n,m \geq N'$, we have $\vert D(y_n,y_m) -D(y,y)\vert < \epsilon$.
Since $\lim\limits_{n \rightarrow \infty} qp_{bl}(y_n,y)=qp_{bl}(y,y)=\lim\limits_{n \rightarrow \infty} qp_{bl}(y,y_n)$. Then for $\epsilon >0$, there exist $N_5$, $N_6 \in \mathbb{N}$ such that 
\[
\vert qp_{bl}(y_n,y) -qp_{bl}(y,y)\vert < \frac{\epsilon}{2} \quad \mbox{for all} \thinspace \thinspace n \geq N_5,
\]
and  
\[
\vert qp_{bl}(y,y_n) -qp_{bl}(y,y)\vert < \frac{\epsilon}{2} \quad \mbox{for all} \thinspace \thinspace n \geq N_6.
\]
Let $N''=\max \{N_5,N_6\}$. Then for all $n \geq N''$, we have $\vert D(y_n,y) -D(y,y)\vert= \vert qp_{bl}(y_n,y)+ qp_{bl}(y,y_n)- 2 qp_{bl}(y,y) \vert < \epsilon$. Therefore, $ \lim\limits_{n,m \rightarrow \infty} D(y_n,y_m) = D(y,y)=\lim\limits_{n \rightarrow \infty} (y_n,y)$. Hence, $(X,D)$ is complete. 
\end{proof}
\section{Fixed Point Results}
Many authors have discussed fixed point theorems on various generalized metric spaces (see \cite{1,2,3,6,8,10,12,13,14}). In this section we obtain some fixed point results in quasi-partial b-metric-like spaces. 
\begin{definition}
\emph{Let $(X,qp_{bl})$ be a quasi-partial b-metric-like space. Then \\
(i) A sequence $\{x_n\} \subseteq X$ is called a \textit{0-Cauchy} sequence if and only if $$\lim\limits_{n,m \rightarrow \infty} qp_{bl}(x_n,x_m) =0=\lim\limits_{n,m \rightarrow \infty} qp_{bl}(x_m,x_n).$$
(ii) A quasi-partial b-metric-like space $(X,qp_{bl})$ is said to be \textit{0-complete} if and only if for every 0-Cauchy sequence $\{x_n\} \subseteq X$, there exists $x \in X$ such that $\lim\limits_{n \rightarrow \infty} qp_{bl}(x_n,x)=\lim\limits_{n \rightarrow \infty} qp_{bl}(x,x_n) =qp_{bl}(x,x)=0=\lim\limits_{n,m \rightarrow \infty} qp_{bl}(x_n,x_m)=\\ \lim\limits_{n,m \rightarrow \infty} qp_{bl}(x_m,x_n).$}
 \end{definition}
 It can be observed that every 0-Cauchy sequence is a Cauchy sequence in a quasi-partial b-metric-like space. Therefore,  every complete quasi-partial b-metric-like space is 0-complete quasi-partial b-metric-like space. 
 \begin{theorem}
Let $(X,qp_{bl})$ be a 0-complete quasi-partial b-metric-like space with coefficient $s \geq 1$. Let $T: X \rightarrow X$ be a map such that
$$qp_{bl}(Tx,Ty) \leq \phi(qp_{bl}(x,y)) \quad \emph{for all} \thinspace \thinspace x,y \in X,$$
where $\phi : [0, \infty) \rightarrow [0,\infty)$ is a continuous map such that $\phi (t)=0$ if and only if $t=0$ and $\phi (t) < t$ for all $t>0$. If $\sum\limits_{n=1}^{\infty} s^n \phi ^n (t)$ converges for all $t>0$ where $\phi ^n$ is the nth iterate of $\phi$. Then $T$ has a unique fixed point.
\end{theorem}
\begin{proof}
Suppose that $x_0 \in X$. We obtain  
\[
qp_{bl}(T^n x_0, T^{n+1} x_0) \leq \phi ^n (qp_{bl}(x_0, T x_0)) \quad \mbox{for all} \thinspace \thinspace  n>1,
\]
and
\[
qp_{bl}(T^{n+1} x_0, T^{n} x_0) \leq \phi ^n (qp_{bl}(T x_0,  x_0)) \quad \mbox{for all} \thinspace \thinspace  n>1.
\]
If $qp_{bl}(T x_0,x_0)=0$ or $qp_{bl}(x_0,T x_0)=0$. Then $T$ has a fixed point. Suppose that $qp_{bl}(x_0,T x_0) >0 $ and $qp_{bl}(Tx_0,x_0) >0$. For $m>n$, we have
\begin{align*}
qp_{bl}(T^n x_0, T^m x_0)  &\leq  s[qp_{bl}(T^n x_0, T^{n+1} x_0) + qp_{bl}(T^{n+1} x_0, T^m x_0)]\\
&\quad -qp_{bl}(T^{n+1} x_0, T^ {n+1} x_0)\\
& \leq  s qp_{bl}(T^n x_0, T^{n+1} x_0)+s qp_{bl}(T^{n+1} x_0, T^m x_0)\\
& \leq  sqp_{bl}(T^n x_0, T^{n+1} x_0)+s [s\{qp_{bl}(T^{n+1}x_0,T^{n+2} x_0)\\
& \quad + qp_{bl}(T^{n+2} x_0, T^m x_0)\} -qp_{bl}(T^{n+2}x_0,T^{n+2}x_0)]\\
 & \leq  s qp_{bl}(T^n x_0, T^{n+1} x_0) +s^2 qp_{bl}(T^{n+1} x_0, T^{n+2} x_0)\\
 & \quad+ s^3 qp_{bl}(T^{n+2} x_0, T^{n+3} x_0)\\
 & \quad +  \ldots +s^{m-n-1} qp_{bl}(T^{m-1} x_0,T^m x_0)\\
 & \leq  s \phi^n (qp_{bl}(x_0, T x_0))+ s^2 \phi^{n+1} (qp_{bl}(x_0, T x_0))\\
 & \quad +s^3 \phi^{n+2} (qp_{bl}(x_0, T x_0))\\
 & \quad + \ldots + s^{m-n-1} \phi^{m-1} (qp_{bl}(x_0, T x_0))\\
 & \leq  \sum\limits_{k=n}^{m-1} s^k \phi ^k (qp_{bl}(x_0,T x_0)).
 \end{align*}
Since $\sum\limits_{n=1}^{\infty} s^n \phi ^n (t)$ converges for all $t>0$. Then $\lim\limits_{n,m \rightarrow \infty} qp_{bl}(T^n x_0, T^m x_0)=0$. Similarly, $ qp_{bl}(T^m x_0, T^n x_0)=0$. Thus, $\{T^n x_0\}$ is a 0-Cauchy sequence. Then there exists $z \in X$ such that
$ \lim\limits_{n \rightarrow \infty} qp_{bl}(T^n x_0,z)= \lim\limits_{n \rightarrow \infty} qp_{bl}(z, T^n x_0)= qp_{bl}(z,z)=0=\lim\limits_{n,m \rightarrow \infty} qp_{bl}(T^n x_0, T^m x_0)=\lim\limits_{n,m \rightarrow \infty} qp_{bl}(T^m x_0, T^n x_0)$.
Since $qp_{bl}(z,Tz)  \leq s[qp_{bl}(z,T^n x_0) + qp_{bl}(T^n x_0,Tz)]-qp_{bl}(T^n x_0, T^n x_0)$. This gives $qp_{bl}(z,Tz) \leq s qp_{bl}(z, T^n x_0)+s \phi(qp_{bl}(T^{n-1} x_0, z))$. Letting $n \rightarrow \infty $ we get, $z=Tz$. Let $z$ and $w$ be two fixed points of $T$. Then $qp_{bl}(z,w)= qp_{bl}(Tz,Tw)$ $\leq \phi (qp_{bl}(z,w)) < qp_{bl}(z,w)$ a contradiction. Thus, $z=w$.
\end{proof}
\begin{corollary}
Let $(X, qp_{bl})$ be a 0-complete quasi-partial b-metric-like space with coefficient $s \geq 1$. Let $T: X \rightarrow X$ be a mapping such that 
$$ qp_{bl}(Tx, Ty) \leq \lambda qp_{bl} (x,y) \quad \emph{ for all} \thinspace \thinspace x,y \in X,$$
where $0 \leq \lambda < \frac{1}{s}$. Then $T$ has a unique fixed point in $X$. Moreover, for any $x_0 \in X$, the iterative sequence $\{T^n x_0\}$ converges to the fixed point.
\end{corollary}
\begin{definition}
\emph{Let $T$ be a self mapping on $X$, then $O(x,T)=\{x,Tx,T^2x, \ldots \}$ is called a orbit of x.}
\end{definition}
\begin{theorem}\label{theorem5.5}
Let $(X,qp_{bl})$ be a quasi-partial b-metric-like space and let $T: X \rightarrow X$. Then the following hold.
\begin{itemize}
\item[(i)] There exists $\phi : X \rightarrow \mathcal{R}^+$ such that 
$$ qp_{bl} (x,Tx) \leq \phi (x)- \phi (Tx) \quad \emph{for all} \thinspace \thinspace x \in X, $$
if and only if  $\sum\limits_{n=0}^{\infty} qp_{bl}(T^n x, T^{n+1} x)$ converges for all $x \in X$.
\item[(ii)] There exists $\phi : X \rightarrow \mathcal{R}^+$ such that 
$$ qp_{bl} (x,Tx) \leq \phi (x)- \phi (Tx) \quad \emph{for all} \thinspace \thinspace x \in O(x), $$
if and only if  $\sum\limits_{n=0}^{\infty} qp_{bl}(T^n x, T^{n+1} x)$ converges for all $x \in O(x)$.
\end{itemize}
\end{theorem}
The proof is similar to the case of quasi-partial b-metric space \cite{8}.
\begin{example}
\emph{Let $X=[0,1]$. Define $qp_{bl}: X \times X \rightarrow [0,\infty)$ as\\
 $qp_{bl}(x,y)=\left\{\begin{array}{cc}
(x+y)^2,& \mbox{if}\thinspace  x \neq y,\\
0, &\mbox{if} \thinspace x=y.
\end{array}
\right. $\\
 Then $(X,qp_{bl})$ is  a quasi-partial b-metric-like space with $s=2$. Define $T: X \rightarrow X$ as $T x = \frac{x}{2}$, then the series $\sum\limits_{n=0}^{\infty} qp_{bl}(T^n x, T^{n+1} x)$ is convergent. We have $\sum\limits_{n=0}^{\infty} qp_{bl}(T^n x, T^{n+1} x) =\sum\limits_{n=0}^{\infty} qp_{bl} \Big( \frac{x}{2^n},\frac{x}{2^{n+1}}\Big)
  = \sum\limits_{n=0}^{\infty} \Big( \frac{x}{2^n}+\frac{x}{2^{n+1}}\Big)=3 x^2$.  Then conditions of Theorem \ref{theorem5.5} are satisfied for $\phi (x)= 3 x^2$.} 
\end{example}
\begin{proposition}
Let $(X, qp_{bl})$ be a quasi-partial b-metric-like space with coefficient $s \geq 1$ and let $\{x_n \}$ be a sequence in $X$ such that $\lim\limits_{n \rightarrow \infty} qp_{bl}(x_n,x)=0= \lim\limits_{n \rightarrow \infty} qp_{bl}(x,x_n)$. Then\\
(i) x is unique.\\
(ii) $\frac{1}{s} qp_{bl}(x,y) \leq \lim\limits_{n \rightarrow \infty} qp_{bl}(x_n,y) \leq s qp_{bl}(x,y)$ for all $y \in X$.
\end{proposition}
\begin{proof}
(i) Suppose that there exists $z \in X$ such that $\lim\limits_{n \rightarrow \infty} qp_{bl}(x_n,z)=0= \lim\limits_{n \rightarrow \infty} qp_{bl}(z,x_n)$. Since $qp_{bl}(x,z) \leq s[qp_{bl}(x,x_n)+ qp_{bl}(x_n,z)]-qp_{bl}(x_n,x_n)$. Therefore, $x=z$.\\
(ii) Since $\frac{1}{s} qp_{bl}(x,y) \leq \frac{1}{s} [s\{qp_{bl}(x,x_n)+ qp_{bl}(x_n,y) \}-qp_{bl}(x_n,x_n)]$. Then $\frac{1}{s} qp_{bl}(x,y) \leq \lim\limits_{n \rightarrow \infty} qp_{bl}(x_n,y)$. Also, $qp_{bl}(x_n,y) \leq s[ qp_{bl}(x_n,x)+qp_{bl}(x,y)]-qp_{bl}(x,x)$. This implies that $\lim\limits_{n \rightarrow \infty} qp_{bl}(x_n,y) \leq s  qp_{bl}(x,y)$.
\end{proof}
\begin{remark}
\emph{Let $(X,qp_{bl})$ be a quasi-partial b-metric-like space with coefficient $s \geq 1$ and if $x \neq y$, then $qp_{bl}(x,y)>0$ and $qp_{bl}(y,x) >0$.}
\end{remark}
\begin{theorem}\label{theorem5.9}
Let $(X,qp_{bl})$ be 0-complete quasi-partial b-metric-like space with coefficient $s \geq 1$ and let $T: X \rightarrow X$ be a map satisfying
\[ \phi ( qp_{bl}(Tx,Ty)) \leq \frac{\phi(qp_{bl}(x,y))}{s}-\psi(qp_{bl}(x,y)) \quad \mbox{for all} \thinspace \thinspace x,y \in X, \]
where $\phi, \psi : [0,\infty) \rightarrow [0, \infty)$ are continuous, monotone non-decreasing functions with $\phi (t)=0=\psi (t)$ if and only if $t=0$. Also $\phi$ is linear and $\phi( \psi (t)) \leq \psi (t)$ for $t>0$. Then $T$ has a unique fixed point.
\end{theorem}
\begin{proof}
Let $x_0 \in X$. Define the sequence $\{x_n \}$ by $x_n=T^n x_0$ for each $n \in \mathbb{N}$. We have $\phi(qp_{bl}(x_n,x_{n+1})) \leq \frac{\phi(qp_{bl}(x_{n-1},x_n))}{s}-\psi(qp_{bl}(x_{n-1},x_n)) \leq \phi(qp_{bl}(x_{n-1},x_n))$. Since $\phi$ is a monotone non-decreasing function. Therefore, $qp_{bl}(x_n,x_{n+1}) \leq qp_{bl}(x_{n-1},x_n)$. This gives $\{qp_{bl}(x_n,x_{n+1})\}$ is a monotone decreasing sequence then there exists $a \geq 0$ such that $\lim\limits_{n \rightarrow \infty} qp_{bl}(x_n,x_{n+1})=a$. Since $\phi(qp_{bl}(x_n,x_{n+1})) \leq \frac{\phi(qp_{bl}(x_{n-1},x_n))}{s}-\psi(qp_{bl}(x_{n-1},x_n))$. 
Letting $n \rightarrow \infty $ and using continuity of $\phi$ and $\psi$  we have $\phi(a) \leq \phi(a)-\psi(a)$. 
Therefore, a=0. Thus,  $\lim\limits_{n \rightarrow \infty} qp_{bl}(x_n,x_{n+1})=0$. Similarly, $\lim\limits_{n \rightarrow \infty} qp_{bl}(x_{n+1},x_n)=0$.
Now we show that $\{x_n\}$ is a 0-Cauchy sequence. For $\epsilon >0$, we can choose $N_1,N_2 \in \mathbb{N}$ such that 
\[ qp_{bl}(x_n,x_{n+1})< \min\Big\{\frac{\epsilon}{2s}, \psi(\frac{\epsilon}{2s})\Big\} \qquad \mbox{for} \thinspace \thinspace n \geq N_1, \]
and
\[ qp_{bl}(x_{n+1},x_n)< \min\Big\{\frac{\epsilon}{2s}, \psi(\frac{\epsilon}{2s})\Big\} \qquad \mbox{for} \thinspace \thinspace n \geq N_2. \]
Choose $N=\max \{N_1,N_2\}$. We claim if $qp_{bl}(x,x_{n_0}) \leq \epsilon$ for $n_0 >N$, then $qp_{bl}(Tx,x_{n_0}) \leq \epsilon$.\\
\textbf{Case-1} If $qp_{bl}(x,x_{n_0}) \leq \frac{\epsilon}{2s}$. We have $qp_{bl}(Tx,x_{n_0}) \leq s[qp_{bl}(Tx,Tx_{n_0}) + qp_{bl}(Tx_{n_0},x_{n_0})]-qp_{bl}(Tx_{n_0},Tx_{n_0})\leq s[qp_{bl}(Tx,Tx_{n_0}) + qp_{bl}(Tx_{n_0},x_{n_0})]$. Since $\phi$ is monotone non-decreasing and linear. Therefore, we have
\begin{align*}
\phi(qp_{bl}(Tx,x_{n_0}))&\leq s\phi(qp_{bl}(Tx,Tx_{n_0}))+s\phi(qp_{bl}(Tx_{n_0},x_{n_0}))\\
&\leq \phi(qp_{bl}(x,x_{n_0}))-s\psi(qp_{bl}(x,x_{n_0})+s \phi(qp_{bl}(Tx_{n_0},x_{n_0}))\\
& \leq \phi\Big( \frac{\epsilon}{2s} \Big)-s\psi(qp_{bl}(x,x_{n_0})+s \phi\Big( \frac{\epsilon}{2s} \Big)\\
& \leq \phi \Big( \frac{\epsilon}{2s} + \frac{\epsilon}{2}\Big).
\end{align*} 
Therefore, $qp_{bl}(Tx,x_{n_0}) \leq \epsilon$.\\
\textbf{Case-2} If $\frac{\epsilon}{2s} \leq qp_{bl}(x,x_{n_0}) \leq \epsilon$. Consider
\begin{align*}
\phi(qp_{bl}(Tx,x_{n_0}))&\leq s\phi(qp_{bl}(Tx,Tx_{n_0}))+s\phi(qp_{bl}(Tx_{n_0},x_{n_0}))\\
&\leq \phi(qp_{bl}(x,x_{n_0}))-s\psi(qp_{bl}(x,x_{n_0})+s \phi(qp_{bl}(Tx_{n_0},x_{n_0}))\\
&\leq \phi(\epsilon)-s \psi\Big( \frac{\epsilon}{2s} \Big)+s \phi \Big( \psi\Big( \frac{\epsilon}{2s} \Big)\Big)\\
&\leq \phi(\epsilon)-s \psi\Big( \frac{\epsilon}{2s} \Big)+s  \psi\Big( \frac{\epsilon}{2s} \Big).
\end{align*}
Therefore, $qp_{bl}(Tx,x_{n_0}) \leq \epsilon$. Thus, the claim is true. Similarly, if  $qp_{bl}(x_{n_0},x)$ $\leq \epsilon$ for $n_0 >N$, then $qp_{bl}(x_{n_0},Tx) \leq \epsilon$. As $qp_{bl}(x_{n_0+1},x_{n_0}) \leq \epsilon$. Therefore, our claim implies that $qp_{bl}(Tx_{n_0+1},x_{n_0}) \leq \epsilon$. Continuing like this we get, $qp_{bl}(x_n,x_{n_0}) \leq \epsilon$ for all $n>n_0$. A similar argument shows $qp_{bl}(x_{n_0},x_m) \leq \epsilon$ for all $m>n_0$. Then for $n,m>N$, we have $qp_{bl}(x_n,x_m) \leq s[qp_{bl}(x_n,x_{n_0})+qp_{bl}(x_{n_0},x_m)]-qp_{bl}(x_{n_0},x_{n_0})\leq s[qp_{bl}(x_n,x_{n_0})+qp_{bl}(x_{n_0},x_m)] \leq 2s \epsilon$. 
Therefore, $\lim\limits_{n,m \rightarrow \infty} qp_{bl}(x_n,x_m)=0$. Similarly, $\lim\limits_{n,m \rightarrow \infty} qp_{bl}(x_m,x_n)=0$. Therefore, there exists $z \in X$ such that $\lim\limits_{n \rightarrow \infty} qp_{bl}(x_n,z)=\lim\limits_{n \rightarrow \infty} qp_{bl}(z,x_n) =qp_{bl}(z,z)=0=\lim\limits_{n,m \rightarrow \infty} qp_{bl}(x_n,x_m)= \lim\limits_{n,m \rightarrow \infty} qp_{bl}(x_m,x_n)$. As $qp_{bl}(z,Tz) \leq s[qp_{bl}(z,x_n)+qp_{bl}(x_n,Tz)]-qp_{bl}(x_n,x_n)\leq s qp_{bl}(z,x_n)+s qp_{bl}(x_n,Tz)]$. Then
\begin{align*}
\phi(qp_{bl}(z,Tz))&\leq s \phi(qp_{bl}(z,x_n))+s \phi(qp_{bl}(x_n,Tz))\\
& \leq s \phi(qp_{bl}(z,x_n))+\phi(qp_{bl}(x_{n-1},z))-s \psi(qp_{bl}(x_{n-1},z))\\
& \leq s \phi(qp_{bl}(z,x_n))+\phi(qp_{bl}(x_{n-1},z)).
\end{align*}
Letting $n \rightarrow \infty $ and using continuity of $\phi$ we have $\phi(qp_{bl}(z,Tz))=0$. Therefore, $z=Tz$. Let $z$ and $w$ be two fixed points of $T$. Then $\phi(qp_{bl}(z,w)) \leq \frac{\phi(qp_{bl}(z,w))}{s}-\psi(qp_{bl}(z,w)) \leq \phi(qp_{bl}(z,w))-\psi(qp_{bl}(z,w))$. Hence, $z=w$.
\end{proof}
\begin{example}
\emph{Let $X=\{0,1,2\}$. Define $qp_{bl}:X\times X \rightarrow [0,\infty)$ as\\
$qp_{bl}(0,0)=0, \hspace{2cm} qp_{bl}(0,1)=2, \hspace{2cm} qp_{bl}(0,2)=6,$\\
$qp_{bl}(1,0)=2, \hspace{2cm} qp_{bl}(1,1)=1, \hspace{2cm} qp_{bl}(1,2)=5,$\\
$qp_{bl}(2,0)=5, \hspace{2cm} qp_{bl}(2,1)=8, \hspace{2cm} qp_{bl}(2,2)=2.$ \\
Then $(X,qp_{bl})$ is a quasi-partial b-metric-like space with $s=\frac{8}{7}$. Let $\phi,\psi:[0,\infty) \rightarrow[0,\infty)$ be defined as $\phi (t)=\frac{t}{2}$ and $\psi (t)=\left\{\begin{array}{cc}
\frac{t^2}{4},& \mbox{if}\thinspace  t \leq 1,\\
\frac{1}{4},&\mbox{if} \thinspace t >1.
\end{array}
\right. $\\
 Define the mapping $T:X \rightarrow X$ as $T0=0$, $T1=0$ and $T2=1$. We observe that \\
$\phi(qp_{bl}(T0,T0))=0 \leq \frac{7}{8}\phi(qp_{bl}(0,0))-\psi(qp_{bl}(0,0)),$\\
$\phi(qp_{bl}(T0,T1))=0 \leq \frac{5}{8}= \frac{7}{8}\phi(qp_{bl}(0,1))-\psi(qp_{bl}(0,1)),$\\
$\phi(qp_{bl}(T0,T2))=1 \leq \frac{19}{8}= \frac{7}{8}\phi(qp_{bl}(0,2))-\psi(qp_{bl}(0,2)),$\\
$\phi(qp_{bl}(T1,T0))=0 \leq \frac{5}{8}= \frac{7}{8}\phi(qp_{bl}(1,0))-\psi(qp_{bl}(1,0)),$\\
$\phi(qp_{bl}(T1,T1))=0 \leq \frac{3}{16}= \frac{7}{8}\phi(qp_{bl}(1,1))-\psi(qp_{bl}(1,1)),$\\
$\phi(qp_{bl}(T1,T2))=1 \leq \frac{31}{16}= \frac{7}{8}\phi(qp_{bl}(1,2))-\psi(qp_{bl}(1,2)),$\\
$\phi(qp_{bl}(T2,T0))=1 \leq \frac{31}{16}= \frac{7}{8}\phi(qp_{bl}(2,0))-\psi(qp_{bl}(2,0)),$\\
$\phi(qp_{bl}(T2,T1))=1 \leq \frac{13}{4}= \frac{7}{8}\phi(qp_{bl}(2,1))-\psi(qp_{bl}(2,1)),$\\
$\phi(qp_{bl}(T2,T2))=\frac{1}{2} \leq \frac{5}{8}= \frac{7}{8}\phi(qp_{bl}(2,2))-\psi(qp_{bl}(2,2)).$\\
All the hypothesis of Theorem \ref{theorem5.9} are satisfied. Then by Theorem \ref{theorem5.9}, $T$ has a unique fixed point. Hence, 0 is the unique fixed point of $T$.}
\end{example}
\begin{lemma}\label{lemma5.11}
Let $(X,qp_{bl})$ be a quasi-partial b-metric-like space with coefficient $s \geq 1$ and $T:X \rightarrow X$. Let $\{x_n \}$ be a sequence in $(X,qp_{bl})$, then for $m>n$ we have
\begin{itemize}
\item[(i)]$qp_{bl}(x_n,x_m) \leq \sum \limits _{k=n}^{m-1} s^k qp_{bl}(x_k,x_{k+1}),$\\
\item[(ii)] $qp_{bl}(x_m,x_n) \leq \sum \limits _{k=n}^{m-1} s^k qp_{bl}(x_{k+1},x_{k}).$
\end{itemize}
\end{lemma}
\begin{proof} (i) Since $(X, qp_{bl})$ is a quasi-partial b-metric-like space with coefficient $s \geq 1$. Therefore, 
 \begin{align*}
qp_{bl}(x_n,x_m) & \leq s[qp_{bl}(x_n,x_{n+1})+ qp_{bl}(x_{n+1},x_m)]-qp_{bl}(x_{n+1},x_{n+1})\\
& \leq s[qp_{bl}(x_n,x_{n+1})+ qp_{bl}(x_{n+1},x_m)]\\
&\leq s qp_{bl}(x_n,x_{n+1})+s^2 qp_{bl}(x_{n+1},x_{n+2})+s^2 qp_{bl}(x_{n+2},x_m)\\
& \quad -s qp_{bl}(x_{n+2},x_{n+2})\\
& \leq s qp_{bl}(x_n,x_{n+1})+s^2 qp_{bl}(x_{n+1},x_{n+2})+s^2 qp_{bl}(x_{n+2},x_m)\\
& \leq s qp_{bl}(x_n,x_{n+1})+s^2 qp_{bl}(x_{n+1},x_{n+2})+s^3 qp_{bl}(x_{n+2},x_{n+3})+\\
& \quad \ldots + s^{m-n-1} qp_{bl}(x_{m-2},x_{m-1})+s^{m-n-1} qp_{bl}(x_{m-1},x_{m})\\
& \leq \sum \limits _{k=n}^{m-1} s^k qp_{bl}(x_k,x_{k+1}).
\end{align*}
(ii) Since $(X, qp_{bl})$ is a quasi-partial b-metric-like space with coefficient $s \geq 1$. Therefore, 
 \begin{align*}
qp_{bl}(x_m,x_n) & \leq s[qp_{bl}(x_m,x_{n+1})+qp_{bl}(x_{n+1},x_n)]-qp_{bl}(x_{n+1},x_{n+1})\\
& \leq s^n qp_{bl}(x_m,x_{n+1})+ s^n qp_{bl}(x_{n+1},x_n)\\
& \leq s^n [s \{ qp_{bl}(x_m,x_{n+2})+qp_{bl}(x_{n+2},x_{n+1}) \}-qp_{bl}(x_{n+2},x_{n+2})]\\
& \quad +s^n qp_{bl}(x_{n+1},x_n)\\
& \leq s^{n+1}qp_{bl}(x_m,x_{n+2})+ s^{n+1}qp_{bl}(x_{n+2},x_{n+1})+s^n qp_{bl}(x_{n+1},x_n)\\
& \leq s^{m-2} qp_{bl}(x_m,x_{m-1})+s^{m-2} qp_{bl}(x_{m-1},x_{m-2})\\
& \quad + \ldots +s^{n+1} qp_{bl}(x_{n+2},x_{n+1})+s^{n} qp_{bl}(x_{n+1},x_{n})\\
& \leq \sum \limits _{k=n}^{m-1} s^k qp_{bl}(x_{k+1},x_k).
\end{align*}
\end{proof}
\begin{lemma}\label{lemma5.12}
Let $(X,qp_{bl})$ be a quasi-partial b-metric-like space with coefficient $s \geq 1$. Let $\{y_n\}$ be a sequence in $(X,qp_{bl})$ such that 
$$qp_{bl}(y_n,y_{n+1}) \leq \lambda [ qp_{bl}(y_{n-1},y_n)+qp_{bl}(y_n,y_{n-1})],$$
and 
$$qp_{bl}(y_{n+1},y_n) \leq \lambda [ qp_{bl}(y_n,y_{n-1})+qp_{bl}(y_{n-1},y_n)],$$
for some $\lambda$ such that $0< \lambda < \frac{1}{2s}$. Then for $m>n$, we have $$\lim\limits_{n,m\rightarrow \infty} qp_{bl}(y_n,y_m) =0 \quad \mbox{and} \quad \lim\limits_{n,m\rightarrow \infty} qp_{bl}(y_m,y_n) =0.$$
\end{lemma}
\begin{proof}
 We have $qp_{bl}(y_n,y_{n+1}) \leq \lambda qp_{bl}(y_{n-1},y_n)+ \lambda qp_{bl}(y_n,y_{n-1})$.\\
  This gives
$qp_{bl}(y_n,y_{n+1}) \leq \lambda[\lambda\{qp_{bl}(y_{n-2},y_{n-1})+qp_{bl}(y_{n-1},y_{n-2})\}]$\\$+ \lambda[\lambda\{qp_{bl}(y_{n-1},y_{n-2})$ $+qp_{bl}(y_{n-2},y_{n-1})\}]$. 
Proceeding likewise we get, $$qp_{bl}(y_n,y_{n+1}) \leq 2^{n-1} \lambda^n [qp_{bl}(y_0,y_1)+qp_{bl}(y_1,y_0)].$$ Then for $m>n$, and using lemma \ref{lemma5.11} we have
\begin{align*}
qp_{bl}(y_n,y_m)&\leq \sum \limits _{k=n}^{m-1} s^k qp_{bl}(y_k,y_{k+1}) \\
& \leq \sum \limits _{k=n}^{m-1} (2s\lambda)^k \Big( \frac{qp_{bl}(y_0,y_1)+qp_{bl}(y_1,y_0)}{2}\Big)\\
& \leq \frac{(2s\lambda)^n}{1-2s\lambda}\Big( \frac{qp_{bl}(y_0,y_1)+qp_{bl}(y_1,y_0)}{2}\Big).
\end{align*}
Since $0< 2s \lambda <1$ therefore, $\lim\limits_{n,m\rightarrow \infty} qp_{bl}(y_n,y_m) =0$. A similar argument shows $\lim\limits_{n,m\rightarrow \infty} qp_{bl}(y_m,y_n) =0$.
\end{proof}
\begin{theorem}
Let $(X,qp_{bl})$ be 0-complete quasi-partial b-metric-like space with coefficient $s \geq 1$. Let $T: X \rightarrow X$ be a surjective map such that 
\begin{align*}
 qp_{bl}(Tx,Ty) &\geq a_1[qp_{bl}(x,y)+qp_{bl}(y,x)]+a_2[qp_{bl}(x,Tx)+qp_{bl}(Tx,x)]\\
 & \quad +a_3[qp_{bl}(y,Ty)+qp_{bl}(Ty,y)]+a_4[qp_{bl}(x,Ty)+qp_{bl}(Ty,x)],
 \end{align*}
for all $x,y \in X$ where $a_i>0$ for each $i=1,2,3,4$ satisfying $1+a_4-a_3 >0$, $s(a_1+a_2)+2s^2(a_3-a_4)+a_4 >2s^2$ and $a_1+a_4 \geq 1$. Then $T$ has a unique fixed point.
\end{theorem}
\begin{proof}
Suppose that $x_0 \in X$. Since $T$ is surjective therefore, there exists $x_1 \in X$ such that $Tx_1=x_0$. Define a sequence $\{x_n\}$ by $x_n=Tx_{n+1}$ for each $n \in \mathbb{N}$. If  $x_{n_0}=x_{n_0+1}$ for some $n_0 \in \mathbb{N}$, then $x_{n_0+1}$ is a fixed point of $T$. Suppose that $x_n \neq x_{n+1}$ for all $n \in \mathbb{N}$. Consider
\begin{align*}
qp_{bl}(x_n,x_{n-1})&\geq a_1[qp_{bl}(x_{n+1},x_n)+qp_{bl}(x_n,x_{n+1})]+a_2[qp_{bl}(x_{n+1},Tx_{n+1})\\
& \quad +qp_{bl}(Tx_{n+1},x_{n+1})] +a_3[qp_{bl}(x_n,Tx_n)+qp_{bl}(Tx_n,x_n)]\\
& \quad +a_4[qp_{bl}(x_{n+1},Tx_n)+qp_{bl}(Tx_n,x_{n+1})]\\
 & = a_1[qp_{bl}(x_{n+1},x_n)+qp_{bl}(x_n,x_{n+1})]+a_2[qp_{bl}(x_{n+1},x_n)\\
 & \quad +qp_{bl}(x_n,x_{n+1})] +a_3[qp_{bl}(x_n,x_{n-1})+qp_{bl}(x_{n-1},x_n)]\\
 & \quad +a_4[qp_{bl}(x_{n+1},x_{n-1})+qp_{bl}(x_{n-1},x_{n+1})].
\end{align*}
Since $qp_{bl}(x_n,x_{n+1}) \leq s[qp_{bl}(x_n,x_{n-1})+qp_{bl}(x_{n-1},x_{n+1})]-qp_{bl}(x_{n-1},x_{n-1})$. This gives $qp_{bl}(x_n,x_{n+1}) \leq s[qp_{bl}(x_n,x_{n-1})+qp_{bl}(x_{n-1},x_{n+1})]$. Therefore, $qp_{bl}(x_{n-1},x_{n+1}) \geq \frac{qp_{bl}(x_n,x_{n+1})-sqp_{bl}(x_n,x_{n-1})}{s}$. Similarly, $qp_{bl}(x_{n+1},x_{n-1}) \geq \frac{qp_{bl}(x_{n+1},x_n)-sqp_{bl}(x_{n-1},x_n)}{s}$. Thus,
\begin{align*}
qp_{bl}(x_n,x_{n-1})&\geq (a_1+a_2)[qp_{bl}(x_{n+1},x_n)+qp_{bl}(x_n,x_{n+1})]\\
& \quad +a_3[qp_{bl}(x_n,x_{n-1})+qp_{bl}(x_{n-1},x_n)]+ \frac{a_4}{s}[qp_{bl}(x_{n+1},x_n)\\
& \quad -sqp_{bl}(x_{n-1},x_n)+qp_{bl}(x_n,x_{n+1})-sqp_{bl}(x_n,x_{n-1})]\\
 &= \Big(a_1+a_2+\frac{a_4}{s}\Big)[qp_{bl}(x_{n+1},x_n)+qp_{bl}(x_n,x_{n+1})]\\
 & \quad +(a_3-a_4)[qp_{bl}(x_n,x_{n-1})+qp_{bl}(x_{n-1},x_n)].
\end{align*}
Therefore, $qp_{bl}(x_{n+1},x_n)+qp_{bl}(x_n,x_{n+1}) \leq \lambda [qp_{bl}(x_n,x_{n-1})+qp_{bl}(x_{n-1},x_n)]$ where $\lambda= \frac{1+a_4-a_3}{a_1+a_2+\frac{a_4}{s}}$. As $1+a_4-a_3 >0$ and $s(a_1+a_2)+2s^2(a_3-a_4)+a_4 >2s^2$,  this gives $0< \lambda < \frac{1}{2s}$. Then by lemma \ref{lemma5.12}, we have $\lim\limits_{n,m\rightarrow \infty} qp_{bl}(x_n,x_m) =0$. Similarly, $\lim\limits_{n,m\rightarrow \infty} qp_{bl}(x_m,x_n) =0$. Therefore, there exists $z\in X$ such that $ \lim\limits_{n \rightarrow \infty} qp_{bl}(x_n,z)=\lim\limits_{n \rightarrow \infty} qp_{bl}(z,x_n)=qp_{bl}(z,z)=0=\lim\limits_{n,m\rightarrow \infty} qp_{bl}(x_n,x_m)=\lim\limits_{n,m\rightarrow \infty} qp_{bl}(x_m,x_n)$. Since $T$ is surjective then there exists $u \in X$ such that $Tu=z$. Consider 
\begin{align*}
qp_{bl}(x_n,z) &\geq a_1[qp_{bl}(x_{n+1},u)+qp_{bl}(u,x_{n+1})]+a_2[qp_{bl}(x_{n+1},Tx_{n+1})\\
&\quad +qp_{bl}(Tx_{n+1},x_{n+1})]+ a_3[qp_{bl}(u,Tu)+qp_{bl}(Tu,u)]\\
&\quad +a_4[qp_{bl}(x_{n+1},Tu)+qp_{bl}(Tu,x_{n+1})]\\
&= a_1[qp_{bl}(x_{n+1},u)+qp_{bl}(u,x_{n+1})]+a_2[qp_{bl}(x_{n+1},x_n)\\
& \quad +qp_{bl}(x_n,x_{n+1})]+ a_3[qp_{bl}(u,z)+qp_{bl}(z,u)]+a_4[qp_{bl}(x_{n+1},z)\\
& \quad +qp_{bl}(z,x_{n+1})].
\end{align*}
Since $qp_{bl}(u,z) \leq s[qp_{bl}(u,x_{n+1})+ qp_{bl}(x_{n+1},z)]-qp_{bl}(x_{n+1},x_{n+1})$. This implies $qp_{bl}(u,z) \leq s[qp_{bl}(u,x_{n+1})+ qp_{bl}(x_{n+1},z)]$. Therefore, $ qp_{bl}(u,x_{n+1}) \geq \frac{qp_{bl}(u,z)-s qp_{bl}(x_{n+1},z)}{s}$. Similarly, $ qp_{bl}(x_{n+1},u) \geq \frac{qp_{bl}(z,u)-s qp_{bl}(z,x_{n+1})}{s}$. Thus, 
\begin{align*}
qp_{bl}(x_n,z) &\geq \frac{a_1}{s}[qp_{bl}(z,u)-sqp_{bl}(z,x_{n+1})+qp_{bl}(u,z)-sqp_{bl}(x_{n+1},z)]\\
& \quad +a_2[qp_{bl}(x_{n+1},x_n)+qp_{bl}(x_n,x_{n+1})]
+ a_3[qp_{bl}(u,z)\\
& \quad+qp_{bl}(z,u)]+a_4[qp_{bl}(x_{n+1},z)+qp_{bl}(z,x_{n+1})]\\
&= \Big( \frac{a_1}{s}+a_3\Big)[qp_{bl}(u,z)+qp_{bl}(z,u)]+a_2 [qp_{bl}(x_{n+1},x_n)\\
& \quad +qp_{bl}(x_n,x_{n+1})] +(a_4-a_1)[qp_{bl}(x_{n+1},z)+qp_{bl}(z,x_{n+1})].
\end{align*}
Letting $n \rightarrow \infty$ we get, $ \Big( \frac{a_1}{s}+a_3\Big)[qp_{bl}(u,z)+qp_{bl}(z,u)] \leq 0$. This implies that $u=z$. Let $z$ and $w$ be two fixed points of $T$. Then 
\begin{align*}
qp_{bl}(z,w) =qp_{bl}(Tz,Tw) &\geq a_1[qp_{bl}(z,w)+qp_{bl}(w,z)]+a_2[qp_{bl}(z,Tz)+\\
& \quad qp_{bl}(Tz,z)]+ a_3[qp_{bl}(w,Tw)+qp_{bl}(Tw,w)]\\
& \quad+a_4[qp_{bl}(z,Tw)+qp_{bl}(Tw,z)]\\
&=(a_1+a_4)[qp_{bl}(z,w)+qp_{bl}(w,z)]+2a_2 qp_{bl}(z,z)
\end{align*}
\begin{align*}
\hspace{2cm} & \quad+2a_3 qp_{bl}(w,w)\\
& \geq (a_1+a_4)[qp_{bl}(z,w)+qp_{bl}(w,z)]\\
& \geq qp_{bl}(z,w)+qp_{bl}(w,z).
\end{align*}
This implies that $z=w$. Hence, $T$ has a unique fixed point in $X$.
\end{proof}
\begin{corollary}\label{corollary5.14}
Let $(X,qp_{bl})$ be 0-complete quasi-partial b-metric-like space with coefficient $s \geq 1$. Let $T: X \rightarrow X$ be a surjective map such that 
 $$qp_{bl}(Tx,Ty) \geq K[qp_{bl}(x,y)+qp_{bl}(y,x)],$$
 for all $x,y \in X$ where $K>2s$. Then $T$ has a unique fixed point.
\end{corollary}
\begin{example}
\emph{Let $X=[0,1]$. Define $qp_{bl}: X \times X \rightarrow [0,\infty)$ as\\
 $qp_{bl}(x,y)=\left\{\begin{array}{cc}
(x+y)^2,& \mbox{if}\thinspace  x \neq y,\\
0, &\mbox{if} \thinspace x=y.
\end{array}
\right.$\\
 Then $(X,qp_{bl})$ is a 0-complete quasi-partial b-metric-like space with $s=2$. Define $T:X \rightarrow X$ by $Tx= 3x \sqrt{1+x^2}$. Clearly $T$ is surjective. Then
 \begin{align*}
qp_{bl}(Tx,Ty)&=\Big(3x \sqrt{1+x^2}+3y \sqrt{1+y^2})^2\\
& \geq (3x+3y)^2\\
&=\frac{9}{2}[(x+y)^2+(y+x)^2].
\end{align*} 
The conditions of Corollary \ref{corollary5.14} are satisfied for $K= \frac{9}{2}$. Then by Corollary \ref{corollary5.14}, $T$ has a unique fixed point. Hence, 0 is the unique fixed point of $T$.}
\end{example}
 \section*{Acknowledgements}
 The corresponding author(Manu Rohilla) is supported by UGC Non-NET fellowship (Ref.No. Sch/139/Non-NET/Math./Ph.D./2017-18/1028).

\end{document}